\documentclass[11pt,reqno,tbtags]{amsart}
\usepackage{amssymb}
\usepackage{natbib}
\bibpunct[, ]{[}{]}{;}{n}{,}{,}

\numberwithin{equation}{section}

\allowdisplaybreaks

\newtheorem{theorem}{Theorem}
\newtheorem{lemma}[theorem]{Lemma}

\newtheorem{corollary}[theorem]{Corollary}

\theoremstyle{definition}

\theoremstyle{remark}

\newenvironment{romenumerate}{\begin{enumerate}
 }{\end{enumerate}}

\newcounter{oldenumi}
{\setcounter{oldenumi}{\value{enumi}}
\begin{romenumerate} \setcounter{enumi}{\value{oldenumi}}}
{\end{romenumerate}}

\newcounter{thmenumerate}
\newenvironment{thmenumerate}
{\setcounter{thmenumerate}{0}%
 \def\item{\par
 \refstepcounter{thmenumerate}\textup{(\roman{thmenumerate})\enspace}}
}
{}

\newcounter{xenumerate}   

\newcommand{\refT}[1]{Theorem~\ref{#1}}

\newcommand{\refL}[1]{Lemma~\ref{#1}}

\newcommand{\refS}[1]{Section~\ref{#1}}

\begingroup
  \divide\count255 by 60
  \multiply\count255 by -60
  \advance\count255 by \time
  \ifnum \count255 < 10 \xdef\klockan{\the\count1.0\the\count255}
  \else\xdef\klockan{\the\count1.\the\count255}\fi
\endgroup

\newcommand\set[1]{\ensuremath{\{#1\}}}

\newcommand\xpar[1]{(#1)}
\newcommand\bigpar[1]{\bigl(#1\bigr)}
\newcommand\Bigpar[1]{\Bigl(#1\Bigr)}

\def\rompar(#1){\textup(#1\textup)}    

\def\xexp(#1){e^{#1}}

\newcommand\floor[1]{\lfloor#1\rfloor}

\newcommand\ntoo{\ensuremath{{n\to\infty}}}

\newcommand\iid{i.i.d.\spacefactor=1000}
\newcommand\ie{i.e.\spacefactor=1000}
\newcommand\eg{e.g.\spacefactor=1000}

\newcommand\whp{whp}

\newcommand{\tend}{\longrightarrow}
\newcommand\dto{\overset{\mathrm{d}}{\tend}}
\newcommand\pto{\overset{\mathrm{p}}{\tend}}

\newcommand\op{o_{\mathrm p}}
\newcommand\Op{O_{\mathrm p}}

\newcounter{CC}
\newcounter{cc}

\newcommand\E{\operatorname{\mathbb E{}}}
\renewcommand\P{\operatorname{\mathbb P{}}}

\newcommand\Po{\operatorname{Po}}
\newcommand\Bin{\operatorname{Bin}}

\newcommand\ga{\alpha}

\newcommand\gl{\lambda}

\newcommand\go{\omega}

\newcommand\eps{\varepsilon}

\newcommand\cK{\mathcal K}

\newcommand\tG{{\widetilde G}}

\def\[#1]{[\![#1]\!]}

\newcommand\qq{^{1/2}}

\newcommand\qqw{^{-1/2}}

\newcommand\qqcw{^{-3/2}}
\newcommand\qqc{^{3/2}}
\newcommand\qw{^{-1}}
\newcommand\qww{^{-2}}

\newcommand\dd{\,\textup{d}}

\newcommand\gnp{\ensuremath{G(n,p)}}

\newcommand\xxxa{X_3^A}
\newcommand\xxxz{X_3^{A*}}
\newcommand\gll{\mu}
\newcommand\gla{\gll_A}
\newcommand\om[1]{;\;#1}
\newcommand\emma{\nu}
\newcommand\ems{n^{\emma}}
\newcommand\wx{W_{\mathrm{max}}}

\def\ij{_{ij}}
\newcommand\pij{p_{ij}}
\newcommand\glij{\gl_{ij}}

\newcommand{\qwga}{^{-\ga}}
\newcommand{\qqga}{^{\ga/2}}
\newcommand{\qqwga}{^{-\ga/2}}
\newcommand{\qiga}{^{1/\ga}}

\newcommand\sqnn{\sqrt{n\log n}}
\newcommand\kqt{\cK_{\mathrm{qt}}}
\newcommand\kgr{\cK_{\mathrm{gr}}}
\newcommand\kmax{\cK_{\mathrm{max}}}
\newcommand\kft{\cK_{\mathrm{ft}}}
\newcommand\gnvs{G_n[V_s^-]}
\newcommand\jw{n^{1-\ga/2}\log^{-\ga/2} n}
\newcommand\ca{(1+o(1))}
\newcommand\caw{(1-o(1))}
\newcommand\pca{(1+\op(1))}
\newcommand\gnq{\overline{G_n}}

\newcommand\ER{Erd\H os--R\'enyi}

\newcommand\REM[1]{{\raggedright\texttt{[#1]}\par\marginal{XXX}}}

\newcommand\urladdrx[1]{{\urladdr{\def~{{\tiny$\sim$}}#1}}}

\begin{document}
\title
{Large cliques in a power-law random graph}

\date{April 28, 2009 (typeset \today{} \klockan)} 

\author{Svante Janson}
\address{Department of Mathematics, Uppsala University, PO Box 480,
SE-751~06 Uppsala, Sweden}
\email{svante.janson@math.uu.se}
\urladdrx{http://www.math.uu.se/~svante/}

\author{Tomasz {\L}uczak}
\address{Faculty of Mathematics and Computer Science,
Adam Mickiewicz University,
ul.~Umultowska 87, 61-614 Pozna\'n, Poland}
\email{tomasz@amu.edu.pl}

\author{Ilkka Norros}
\address{VTT Technical Research Centre of Finland,
P.O. Box 1000, 02044 VTT,
Finland}
\email{ilkka.norros@vtt.fi}

\thanks{This research was done at Institut Mittag-Leffler,
Djursholm, Sweden, during the program `Discrete Probability' 2009.
The second author partially supported by the Foundation
for Polish Science.}


\date{April 29, 2009}

\begin{abstract}
We study the size of the largest clique $\omega(G(n,\ga))$
in a random graph $G(n,\ga)$ on $n$ vertices which
has power-law degree distribution with exponent $\ga$.
We show that for `flat' degree
sequences with $\ga>2$ whp the largest clique in $G(n,\ga)$ is of a constant
size, while for the heavy tail distribution,
when $0<\ga<2$,  $\omega(G(n,\ga))$ grows as a power of $n$.
Moreover, we show that a natural simple algorithm
whp  finds in $G(n,\ga)$ a large clique
of size $(1+o(1))\omega(G(n,\ga))$ in polynomial time.
\end{abstract}

\maketitle

\section{Introduction}\label{S:intro}

Random graphs with finite density and
power-law degree sequence have attracted
much attention for the last few years
(\eg\ see Durrett~\cite{Du}
and the references therein).
Several models for such graphs has been proposed; in this
paper we concentrate on a Poissonian model
$G(n,\ga)$ in which the number of  vertices of degree at least
$i$ decreases roughly  as $n i^{-\ga}$ (for a precise definition of the model
see Section~\ref{Smodel} below).

We show (Theorem~\ref{T1}) that there is a major difference in the
size of the largest clique $\omega(G(n,\ga))$ between the cases
$\ga<2$ and $\ga>2$  with an intermediate result for $\ga=2$.
In the 'light tail case', when $\ga>2$ (this is when the asymptotic
degree distribution has a finite second moment), the size of the largest
clique is either two or three, i.e., it is almost the same as in
the standard binomial model of random graph $G(n,p)$
in which the expected average degree
is a constant. As opposite to that, in the `heavy tail case',
when $0<\ga<2$, $\omega(G(n,\ga))$ grows roughly as $n^{1-\ga/2}$.
In the critical case when $\ga=2$ we have $\omega(G(n,\ga))=\Op(1)$, but
the probability that $G(n,\ga)\ge k$ is bounded away from zero
for every~$k$.
We also show (Corollary~\ref{coralg}) that in each
of the above cases there exists a simple
algorithm which whp finds in $G(n,\ga)$ a clique of size
$(1-o(1))\omega(G(n,\ga))$. This is quite different from
the binomial case, where it is widely believed that
finding large clique is hard (see for instance Frieze and McDiarmid~\cite{FMcD}).

Similar but less precise results have been
obtained by Bainconi and Marsili~\cite{BianconiM,BianconiM2}
for a slightly different model (see Section~\ref{subsec} below).

\section{The model and the results}\label{Smodel}

The model we study is a version of the conditionally Poissonian
random graph studied by
Norros and Reittu \cite{NorrosR}
(see also Chung and Lu \cite{ChungLu} for a related model).
For $\ga>1$  it is also  an
example of the `rank 1 case' of inhomogeneous random graph studied
by Bollob\'as, Janson, and Riordan~\cite[Section 16.4]{SJ178}.

In order to define our model
consider  a set of $n$ vertices (for convenience labelled
$1,\dots,n$). We first assign a \emph{capacity} or
\emph{weight} $W_i$ to each vertex~$i$. For
definiteness and simplicity, we assume that these are \iid{}
random variables with a distribution with a power-law tail
\begin{equation}
  \label{m1}
\P(W>x)=ax\qwga,
\qquad x\ge x_0,
\end{equation}
for some constants $a>0$ and $\ga>0$, and some
$x_0>0$ (here and below we write $W$ for any of the $W_i$ when the
index does not matter).
Thus, for example, $W$ could have a Pareto
distribution, when $x_0=a\qiga$ and $\P(W>x)=1$ for
$x\le x_0$, but the distribution could be arbitrarily
modified for small $x$.
We denote the largest weight by  $\wx=\max_i W_i$.
Observe that \eqref{m1} implies
\begin{equation}\label{m3}
  \P(\wx>tn\qiga)\le n\P(W>tn\qiga) = O(t\qwga).
\end{equation}
Note also that $\E W^\beta<\infty$
if and only if $\ga>\beta$; in particular, for the `heavy tail case'
when $\ga<2$ we have $\E W^2=\infty$.

Now, conditionally given the weights
$\set{W_i}_1^n$, we join each pair \set{i,j} of
vertices by $E_{ij}$ parallel edges, where the
numbers $E_{ij}$ are independent Poisson distributed
random numbers with means
\begin{equation}\label{glij}
  \E E_{ij}=\gl_{ij}=b\frac{W_iW_j}n,
\end{equation}
where $b>0$ is another constant. We denote the resulting
random (multi)graph by $\hat G(n,\ga)$. For our purposes, parallel edges
can be merged into a single edge, so we may alternatively define
$G(n,\ga)$ as the random simple graph where vertices $i$ and
$j$ are joined by an edge with probability
\begin{equation}\label{pij}
  p_{ij}=1-\exp(-\gl_{ij}),
\end{equation}
independently for all pairs $(i,j)$ with $1\le i<j\le n$.

Then our main result can be stated as follows.
Let us recall that an event holds
\emph{with high probability} (whp), if
it holds with probability tending to 1 as $n\to\infty$.
We also use $\op$ and $\Op$ in the standard sense (see,
for instance, Janson, {\L}uczak, Ruci\'nski~\cite{JLR}).

  \begin{theorem}\label{T1}
	\begin{thmenumerate}
\item\label{T1<2}
If\/ $0<\ga<2$, then
\begin{equation*}
  \omega(G(n,\ga))=(c+\op(1)) n^{1-\ga/2}(\log n)^{-\ga/2},
\end{equation*}
where
\begin{equation}\label{c}
  c=a b\qqga(1-\ga/2)\qqwga.
\end{equation}
\item\label{T1=2}
If\/ $\ga=2$, then $\omega(G(n,\ga))=\Op(1)$; that is, for
every $\eps>0$ there exists a constant $C_\eps$ such
that $\P(\omega(G(n,\ga))>C_\eps)\le\eps$ for every
$n$. However, there is no fixed finite bound $C$ such that
$\omega(G(n,\ga))\le C$ whp.
\item\label{T1>2}
If\/ $\ga>2$, then $\omega(G(n,\ga))\in\set{2,3}$ whp.
Moreover, the probabilities of each of the events $\go(G(n,\ga))=2$
and $\go(G(n,\ga))=3$ tend to positive limits, given by \eqref{t1>2lim}.
	\end{thmenumerate}
  \end{theorem}

A question which naturally emerges when studying
the size of the largest clique in a class of graphs is
whether one can find a large clique in such graph
in a polynomial time. By Theorem~\ref{T1}, whp one can
find $\omega(G(n,\ga))$ in a polynomial time
for $\ga>2$,  and, with some extra effort, the same can be accomplished
for $\ga=2$ (see Corollary~\ref{coralg}).
Thus let us concentrate for the case $\ga<2$, when the large clique
is of polynomial size.
Let us suppose that we know the vertex weights $W_i$ defined in
\refS{Smodel} and, for simplicity, that these are
distinct (otherwise we resolve ties randomly; we omit the details).
Since vertices with larger weights tend to have higher degrees, they
are more likely to be in a large clique, so it is natural to try to
find a large clique by looking at the vertices with largest
weights. One simple way is the greedy algorithm which checks the
vertices in order of decreasing weights and selects every vertex that
is joined to every other vertex already selected. This evidently
yields a clique, which we call the \emph{greedy clique} and
denote by $\kgr$. Thus
\begin{equation*}
  \kgr=\set{i:i\sim j \text{ for all $j$
  with } W_j>W_i \text{ and } j\in\kgr}.
\end{equation*}

A simplified algorithm is to select every vertex that is joined to
every vertex with higher weight, regardless of whether these already
are selected or not. This gives the \emph{quasi top clique}
studied by Norros~\cite{Norros09}, which we denote by
$\kqt$. Thus
\begin{equation*}
  \kqt=\set{i:i\sim j \text{ for all $j$ with } W_j>W_i}.
\end{equation*}
Obviously, $\kqt\subseteq\kgr$.
The difference between the two cliques is that if we, while checking
vertices in order of decreasing weights, reject a vertex, then that
vertex is ignored for future tests when constructing $\kgr$,
but not for $\kqt$.
A more drastic approach is to stop at the first failure; we define the
\emph{full top clique} $\kft$ as the result, \ie
\begin{equation*}
  \kft=\set{i:j\sim k \text{ for all distinct $j,k$
  with } W_j,W_k\ge W_i}.
\end{equation*}
Thus $\kft$ is the largest clique consisting of all vertices
with weights in some interval $[w,\infty)$. Clearly,
$\kft\subseteq\kqt\subseteq\kgr$.
Finally, by  $\kmax$ we denote  the largest clique (chosen at
random, say, if there is a tie).
Thus
\begin{equation}\label{kkkk}
  |\kft|\le|\kqt|\le|\kgr|\le|\kmax|=\go(G(n,\ga)).
\end{equation}
The following theorem shows that the last two inequalities in
\eqref{kkkk} are
asymptotic equalities, but not the first one.
Here we use $\pto$ for convergence in probability,  and
all unspecified limits are as \ntoo.

\begin{theorem}
  \label{T2}
If $0<\ga<2$, then  $G(n,\ga)$, $\kgr$ and $\kqt$ both have size
$\pca\go(G(n,\ga))$;
in other words
\begin{equation*}
  |\kgr|/|\kmax|\pto1
\quad\text{and}\quad
  |\kqt|/|\kmax|\pto1.
\end{equation*}
On the other hand,
\begin{equation*}
  |\kft|/|\kmax|\pto2\qqwga.
\end{equation*}
\end{theorem}

Thus, whp $\kgr$ and $\kqt$ almost attain the
maximum size of a clique, while $\kft$ falls short by a
constant factor. As a simple corollary of the above result
one can get the following.

\begin{corollary}\label{coralg}
For every $\ga>0$ there exists an algorithm which whp
finds in $G(n,\ga)$ a clique of size $(1+o(1))\omega(G(n,\ga))$
in a polynomial time.
\end{corollary}

\section{The proof for the case $\ga<2$ (no second moment)}\label{S<2}

We begin with a simple lemma giving an upper bound for the clique
number of the \ER{} random graph $\gnp$ (for much more precise results
see, for instance, Janson, {\L}uczak, Ruci\'nski~\cite{JLR}).

\begin{lemma}
  \label{Lgnp}
For any $p=p(n)$, whp
\begin{equation*}
  \go(\gnp)
\le \frac{2\log n}{1-p}.
\end{equation*}
\end{lemma}

\begin{proof}
Denote by $X_k$ the number of cliques of order $k$ in $\gnp$.
For the expected number of such cliques
we have
\begin{equation*}
	\E X_k
= \binom nk p^{\binom k2}
\le\Bigpar{\frac{ne}k p^{(k-1)/2}}^k.
  \end{equation*}
If we set $k\ge\floor{2\log (n)/(1-p)}$, then
$$p^{(k-1)/2}=\big(1-(1-p)\big)^{(k-1)/2}\le e^{-(1-p)(k-1)/2}\le e/n\,.$$
Consequently,
we arrive at
  \begin{equation*}
\P(\go(\gnp)\ge k)
=\P(X_k\ge1)
\le	\E X_k
\le\Bigpar{\frac{e^{2}}k }^k\to0,
\end{equation*}
since $k\ge \floor{2\log n}\to\infty$.
\end{proof}

\begin{proof}[Proof of Theorems \ref{T1}\ref{T1<2} and \ref{T2}]
	
For $s>0$, let us  partition the vertex set
$V=\set{1,\dots,n}$ of $G_n=G(n,\ga)$ into
\begin{align*}
  V_s^-=\set{i:W_i\le s\sqrt{n\log n}}
\qquad\text{and}\qquad
  V_s^+=\set{i:W_i> s\sqrt{n\log n}};
\end{align*}
we may think of elements of $V_s^-$ and $V_s^+$
as `light' and `heavy' vertices, respectively.
By \eqref{m1},
\begin{equation}
 \E |V_s^+|= n\P(W>s\sqnn)=as\qwga \jw.
\end{equation}
Moreover, $|V_s^+|\sim\Bin(n,\P(W>s\sqnn))$, and
Chebyshev's inequality (or the sharper Chernoff
bounds~\cite[Section 2.1]{JLR}) implies that whp
\begin{equation}\label{jb3}
  |V_s^+| = \ca\E|V_s^+| =\ca as\qwga \jw.
\end{equation}

We now condition on the sequence of weights \set{W_k}.
We will repeatedly use the fact that if $i$ and $j$ are
vertices with weights $W_i=x\sqnn$ and $W_j=y\sqnn$,
then by \eqref{glij}--\eqref{pij},
$\glij=bxy\log n$ and
\begin{equation}
  \label{jep}
\pij=1-e^{-\glij}=1-e^{-bxy\log n}=1-n^{-bxy}.
\end{equation}
In particular,
\begin{align}
 \pij&\le 1-n^{-bs^2}, \qquad  \text{if }  i,j\in V_s^-,
\label{jep-}
\\
\pij&> 1-n^{-bs^2}, \qquad  \text{if } i,j\in V_s^+.
\label{jep+}
\end{align}

Consider, still conditioning on \set{W_k},
for an $s$ that will be chosen later,
the induced subgraph $\gnvs$ of $G(n,\ga)$ with vertex set $V_s^-$.
This graph has at most $n$ vertices and, by \eqref{jep-},
all edge probabilities are at most $1-n^{-bs^2}$, so we may
regard $\gnvs$ as a subgraph of $\gnp$ with
$p=1-n^{-bs^2}$. Hence, \refL{Lgnp} implies that
whp
\begin{equation}
  \label{jb1}
\go(\gnvs) \le \frac{2\log n}{n^{-bs^2}} = 2n^{bs^2}\log n.
\end{equation}
If $\cK$ is any clique in $G(n,\ga)$, then $\cK\cap V_s^-$ is a clique in
$\gnvs$, and thus
$|\cK\cap V_s^-|\le\go(\gnvs)$; further, trivially,
$|\cK\cap V_s^+|\le |V_s^+|$. Hence,
$|\cK|\le\go(\gnvs)+|V_s^+|$, and thus
\begin{equation}\label{jb2}
  \go(G(n,\ga))\le\go(\gnvs)+|V_s^+|.
\end{equation}
We choose, for a given $\eps>0$,
$s=(1-\eps)b\qqw(1-\ga/2)\qq$ so that the exponents
of $n$ in \eqref{jb1} and \eqref{jb3} almost
match; we then obtain from \eqref{jb2}, \eqref{jb3},
and  \eqref{jb1}, that whp
\begin{equation}\label{jb4}
  \begin{split}
  \go(G(n,\ga))
&\le \ca as\qwga \jw
\\&
= \ca (1-\eps)\qwga c\jw,	
  \end{split}
\end{equation}
with $c$ defined as in (\ref{c}).

To obtain a matching lower bound, we consider the quasi top
clique $\kqt$.
Let, again, $s$ be fixed and condition on the weights \set{W_k}.
If $i,j\in V_s^+$, then by \eqref{jep+}, the probability that $i$ is
not joined to $j$ is less than $n^{-bs^2}$. Hence,
conditioned on the weights \set{W_k}, the probability that a
given vertex $i\in V_s^+$ is {not} joined to
every other $j\in V_s^+$ is at most
$|V_s^+|n^{-bs^2}$, which by \eqref{jb3} whp
is at most
$2as\qwga n^{1-\ga/2-bs^2}\log\qqwga n$.
We now choose $s=(1+\eps)b\qqw(1-\ga/2)\qq$ with
$\eps>0$. Then,
for some constant $C<\infty$, whp
\begin{equation*}
  \P\bigpar{i\notin\kqt\mid\set{W_k}}
\le C n^{-2\eps(1-\ga/2)}
\end{equation*}
and thus
\begin{equation*}
  \E\bigpar{|V_s^+\setminus\kqt|\mid\set{W_k}}
\le C n^{-2\eps(1-\ga/2)}|V_s^+|.
\end{equation*}
Hence, by Markov's inequality, whp
\begin{equation*}
  |V_s^+\setminus\kqt|
\le C n^{-\eps(1-\ga/2)}|V_s^+|.
\end{equation*}
Thus, using \eqref{jb3}, whp
\begin{equation}\label{jb5}
  \begin{split}
\go(G(n,\ga))
\ge|\kgr|
\ge|\kqt|
&\ge|V_s^+|-  |V_s^+\setminus\kqt|
\ge \caw|V_s^+|
\\&
\ge\caw as\qwga \jw
\\&
= \ca (1+\eps)\qwga c\jw.	
  \end{split}
\end{equation}
Since $\eps>0$ is arbitrary, \eqref{jb4} and
\eqref{jb5} imply \refT{T1}\ref{T1<2}
and the first part of \refT{T2}.

In order to complete the proof
of \refT{T2}, it remains to consider $\kft$.
Define $\gnq$ as the complement of
$G(n,\ga)$. Then, using \eqref{jep+} and conditioned on
$\set{W_i}$, we infer that
the expected number of edges of $\gnq$ with both endpoints in
$V_s^+$ is at most $n^{-bs^2}|V_s^+|^2$.
If we choose $s=b\qqw(2-\ga)\qq$, then \eqref{jb3} implies that
this is whp $o(1)$; hence whp $V_s^+$ contains no edges of
$\gnq$, \ie, $\kft\supseteq V_s^+$.

On the other hand, let 
$0<\eps<1/2$ and define, still with $s=b\qqw(2-\ga)\qq$,
$V'=V_{(1-2\eps)s}^+\cap V_{(1-\eps)s}^-$.
Then, conditioned on \set{W_i},
the probability of having no edges of
$\gnq$ in $V'$ is,
by \eqref{jep-},
\begin{equation}\label{er1}
  \begin{split}
\prod_{i,j\in V'} \pij
\le
\bigpar{1-n^{-b(1-\eps)^2s^2}}^{\binom{|V'|}{2}}
\le
\exp\bigpar{-n^{-(1-\eps)^2(2-\ga)}(|V'|-1)^2/2}
.	
  \end{split}
\end{equation}
By \eqref{jb3}, whp
\begin{equation*}
  \begin{split}
|V'|-1
&=|V_{(1-2\eps)s}^+|-|V_{(1-\eps)s}^+|-1
\\
&= \ca a\bigpar{(1-2\eps)\qwga-(1-\eps)\qwga}s\qwga \jw,
  \end{split}
\end{equation*}
and it follows from \eqref{er1} that
\begin{equation*}
  \P(\kft\supseteq V_{(1-2\eps)s}^+)
\le \prod_{i,j\in V'}\pij \to0.
\end{equation*}
Hence, whp $\kft\subset V_{(1-2\eps)s}^+$.

We have shown that, for any $\eps\in(0,1/2)$,
whp $V_s^+\subseteq \kft\subset
V_{(1-2\eps)s}^+$, and it follows by \eqref{jb3} and \eqref{c}
(by letting $\eps\to0$) that whp
\begin{equation*}
  \begin{split}
  |\kft|
&=\ca|V_s^+|  =\ca as\qwga \jw
\\&
=\ca 2\qqwga c\jw
=\ca 2\qqwga \go(G(n,\ga)),
  \end{split}
\end{equation*}
where the last equality follows from \refT{T1}.
\end{proof}

\section{The case $\ga=2$ (still no second moment)}\label{S=2}

\begin{proof}[Proof of \refT{T1}\ref{T1=2} and Corollary~\ref{coralg}]
Given the weights $W_i$, the probability that four vertices
$i,j,k,l$ form a clique is, by \eqref{pij} and \eqref{glij},
\begin{equation*}
  p_{ij}p_{ik}p_{il}p_{jk}p_{jl}p_{kl}
\le
  \gl_{ij}\gl_{ik}\gl_{il}\gl_{jk}\gl_{jl}\gl_{kl}
=b^6\frac{W_i^3W_j^3W_k^3W_l^3}{n^6}.
\end{equation*}
Thus, if $X_m$ is the number of cliques of size $m$ in
$G(n,\ga)$, then the conditional expectation of $X_4$ is
\begin{equation}\label{ml2}
  \E( X_4\mid \set{W_i}_1^n)
\le b^6n^{-6}\sum_{i<j<k<l} W_i^3W_j^3W_k^3W_l^3
\le b^6 \Bigpar{n\qqcw\sum_i W_i^3}^4.
\end{equation}
To show that the number of such quadruples is bounded in probability,
we shall calculate a truncated expectation of $\sum_iW_i^3$.
Using \eqref{m1}, for any constant $A>0$, we get
\begin{equation}\label{ika}
\begin{split}
    \E\Bigpar{\sum_i W_i^3 \om{\wx\le An\qq}}
&\le
\E\Bigpar{\sum_i \min(W_i,An\qq)^3}
\\&
= n\E \min(W,An\qq)^3
\\&
= n\int_0^{An\qq} 3x^2\P(W>x)\dd x
\\&
=O(n A n\qq),
\end{split}
\end{equation}
and thus, using \eqref{m3} and  Markov's inequality,
for every $t>0$ and some constant
$C$ independent of $A$, $t$ and $n$, we arrive at
\begin{align}
\P\Bigpar{n\qqcw&\sum_i W_i^3>t}\notag
\\&
\le
t\qw \E\Bigpar{n\qqcw\sum_i W_i^3 \om{\wx\le An\qq}}
+\P(\wx>An\qq) \notag
\\&
\le CA t\qw  + C A\qww.	 \label{magn}
\end{align}
Given $t>0$, we choose $A=t^{1/3}$ and find
$\P\Bigpar{n\qqcw\sum_i W_i^3>t} = O(t^{-2/3})$. Hence,
$n\qqcw\sum_i W_i^3=\Op(1)$, and it follows by
\eqref{ml2} and Markov's inequality that $X_4=\Op(1)$.

Finally, we observe that, for any $m\ge4$,
\begin{equation}
  \P(\omega(G(n,\ga))\ge m)
\le \P\Bigpar{X_4\ge\binom m4},
\end{equation}
which thus can be made arbitrarily small (uniformly in $n$) by
choosing $m$ large enough. Hence, $\omega(G(n,\ga))=\Op(1)$.

To complete the proof of Theorem~\ref{T1}(ii) let us note that for any fixed
$m\le n$, the  probability that there are at least $m$ vertices
with weights $W_i>n\qq$ is larger than
$c_1>0$ for some absolute constant $c_1>0$,
and conditioned on this event, the probability
that the $m$ first of these vertices form a clique
is larger than $c_2$ for some absolute constants
$c_1,c_2$ not depending on $n$.

Finally, we remark that all cliques of size four can clearly be found
in time $O(n^4)$. The number of such cliques is whp at most $\log\log
n$, say,  so there exists an algorithm
which whp finds the largest clique in a polynomial time (for example
by crudely checking all sets of cliques of size 4).
\end{proof}

\section{$\ga>2$ (finite second moment)}\label{S>2}

\begin{proof}[Proof of \refT{T1}\ref{T1>2}]
Choose  $\emma$ such that $1/2>\emma>1/\ga$. Then \eqref{m3}
(or \eqref{m1} directly) implies that whp $\wx\le \ems$.
Furthermore, in analogy to \eqref{ika} and \eqref{magn},
\begin{multline}
    \E\Bigpar{\sum_i W_i^3\om{\wx\le \ems}}
\le
\E\Bigpar{\sum_i \min(W_i,\ems)^3}
\\
= n\int_0^{\ems} 3x^2\P(W>x)\dd x
=O(n \ems)=o(n\qqc),
\end{multline}
and thus
\begin{equation}  \label{erika}
\begin{split}
\P\Bigpar{n\qqcw&\sum_i W_i^3>t}
\\&
\le
t\qw \E\Bigpar{n\qqcw\sum_i W_i^3\om{\wx\le \ems}}
+\P(\wx>\ems)
\\&
=o(1).
\end{split}\end{equation}
Hence, $n\qqcw\sum_i W_i^3\pto0$, and it follows from
\eqref{ml2} that
\begin{equation*}
  \P(\omega(G(n,\ga))\ge 4)=\P(X_4\ge1)
\le \E\bigpar{\min\bigpar{1,  \E( X_4\mid \set{W_i}_1^n)}} \to0.
\end{equation*}
Consequently, whp $\omega(G(n,\ga))\le 3$.

Moreover, we can similarly estimate
\begin{equation}\label{sjw}
  \E X_3
\le\E\sum_{i<j<k}\gl_{ij}\gl_{ik}\gl_{jk}
=
\E\Bigpar{b^3n^{-3}\sum_{i<j<k}W_i^2W_j^2W_k^2}
\le \tfrac16 b^3 \bigpar{\E W^2}^3;
\end{equation}
note that $\E W^2<\infty$ by \eqref{m1} and the
assumption $\ga>2$. Hence, the number of $K_3$ in
$G(n,\ga)$ is $\Op(1)$. To obtain the limit distribution, it is
convenient to truncate the distribution, as we have
 done it in the previous section. We let $A$ be a fixed large constant,
and let $\xxxa$ be the number of $K_3$ in $G(n,\ga)$ such
that all three vertices have weights at most $A$, and let
$\xxxz$ be the number of the remaining triangles.
Arguing as in \eqref{sjw}, we see easily that
\begin{align}\label{sjw1}
  \E \xxxa
&\le \tfrac16 b^3 \bigpar{\E (W^2\om{W\le A})}^3
  \\
\label{sjw2}
  \E \xxxz
&\le  b^3 \bigpar{\E (W^2)}^2\E (W^2\om{W> A}).
\end{align}
Moreover, if $W_i,W_j\le A$, then
$\gl_{ij}=O(1/n)$, and thus
$\pij\sim\glij$, and it is easily seen that
\eqref{sjw1} can be sharpened to
\begin{align}
  \E \xxxa
\to \gla=\tfrac16 b^3 \bigpar{\E (W^2\om{W\le A})}^3.
\end{align}
Furthermore, we may calculate fractional moments
$\E(\xxxa)_m$
by the same method, and it follows easily by a standard argument
(see, for instance, \cite[Theorem~3.19]{JLR} for
$G(n,p)$) that $\E(\xxxa)_m\to \gla^m$ for
every $m\ge1$, and thus
by the method of moments \cite[Corollary 6.8]{JLR}
\begin{equation}
  \label{sjwA}
\xxxa\dto\Po(\gla)
\end{equation}
as $\ntoo$, for every
fixed $A$, where $\dto$ denotes the convergence in distribution.

Finally, we note that
the right hand side of \eqref{sjw2} can be made
arbitrarily small by choosing $A$ large enough, and hence
\begin{align}
  \lim_{A\to\infty} \limsup_{\ntoo} \P(\xxxz\neq0)=0,
\end{align}
and that $\gla\to\gll=\frac16 \xpar{b\E (W^2)}^3$
as $A\to\infty$. It follows by a standard argument (see
Billingsley~\cite[Theorem 4.2]{Bill}) that we can let
$A\to\infty$ in \eqref{sjwA} and obtain
\begin{equation}
  X_3\dto \Po(\gll).
\end{equation}
In particular, $\P(X_3=0)\to e^{-\gll}$,
which yields the following result:
\begin{equation}\label{t1>2lim}
  \begin{aligned}
\P(\omega(G(n,\ga))=2)&\to e^{-\gll}
=
e^{-\frac16 \xpar{b\E (W^2)}^3},
\\
\P(\omega(G(n,\ga))=3)&\to1- e^{-\gll}
=
1-e^{-\frac16 \xpar{b\E (W^2)}^3}.
  \end{aligned}
\end{equation}
Finally, note that  $G(n,\ga)$ whp contains cliques $K_2$,
\ie{} edges, so clearly $\omega(G(n,\ga))\ge2$.
\end{proof}

\section{Final remarks}

In this section we make some comments on other models
of power-law random graphs as well as some remarks on
possible variants of our results.
We omit detailed proofs.

Let us remark first that,  for convenience
 and to facilitate comparisons with
other papers, in the definition of $G(n,\ga)$ we used two scale
parameters $a$ and $b$ above, besides the important
exponent $\ga$. By rescaling $W_i\mapsto t W_i$
for some fixed $t>0$, we obtain the same $G(n,\ga)$ for the
parameters $at^\ga$ and $bt\qww$; hence only the
combination $ab\qqga$ matters, and we could fix either
$a$ or $b$ as 1 without loss of generality.

\subsection{Algorithms based on degrees}
As for the algorithmic result Theorem~\ref{T2}, it remains
true if we search for large cliques examing the vertices one by one
in order not by their weights but by their degrees and modify the definition
of  $\kgr$, $\kqt$, and $\kft$ accordingly.
(This holds both if we take the degrees in the multigraph $\hat
G(n,\ga)$, or if we consider the corresponding simple graph.)
The reason is that, for the vertices of large weight that we are
interested in, the degrees are whp all almost proportinal to the
weights, and thus the two orders do not differ very much.
This enables us to find an almost maximal clique in polynomial time,
even without knowing the weights.

\subsection{More general weight distributions}
Observe  that Theorems~\ref{T1} and~\ref{T2}
remain true (and can be shown by basically the same argument),
provided only that
the power law holds asymptotically for large weights, i.e., \eqref{m1} may be
relaxed to
\begin{equation}
  \label{m1x}
\P(W>x)\sim ax\qwga \qquad\text{as } x\to\infty.
\end{equation}

\subsection{Deterministic weights}
Instead of choosing weights independently according to
 the distribution $W$ we may as well take
 a suitable deterministic sequence $W_i$ of weights
(as in Chung and Lu~\cite{ChungLu}), for example
\begin{equation}\label{rvfixed}
  W_i=a\qiga\frac{n\qiga}{i\qiga},
\qquad i=1,\dots,n.
\end{equation}
All our results remain true also in this setting; in fact
the proofs are slightly simpler for this model. A particularly
interesting special case for this model
(see Bollob\'as, Janson, and Riordan~\cite[Section 16.2]{SJ178} and Riordan~\cite{Rsmall}) is when  $\ga=2$, where
\eqref{glij} and \eqref{rvfixed} combine to yield
\begin{equation*}
  \glij=\frac{ab}{\sqrt{ij}}.
\end{equation*}

\subsection{Poisson number of vertices}
We may also let the number of vertices be random with a Poisson $\Po(n)$
distribution (as in, e.g., Norros~\cite{Norros09}). Then the set of weights
$\set{W_i}_1^n$ can be regarded as a Poisson process on
$[0,\infty)$ with intensity measure $n\dd\mu$,
  where $\mu$ is the distribution of the random variable
  $W$ in \eqref{m1}. Note that now $n$ can be any positive real number.

\subsection{Different normalization}
A slightly different  power-law random graph model emerges when
instead of  \eqref{glij} we define the intensities
$\gl_{ij}$ by
\begin{equation}
  \gl_{ij}=\frac{W_iW_j}{\sum_{k=1}^n W_k}
\end{equation}
(see for instance Chung and Lu~\cite{ChungLu} and
Norros and Reittu~\cite{NorrosR}).
Let us call this model $\tG(n,\ga)$.
In the case $\ga>1$, when the mean $\E W<\infty$,
the results for $\tG(n,\ga)$ and $G(n,\ga)$  are not much different.
In fact, by the law of large numbers, $\sum_1^n W_k/n\to\E
W$ a.s., so we may for any $\eps>0$ couple $\tG(n,\ga)$
constructed by this model with
$G(n,\ga)^\pm$ constructed as above, using
\eqref{glij} with $b=1/(\E W\mp\eps)$, such
that whp $G(n,\ga)^-\subseteq \tG(n,\ga)\subseteq G(n,\ga)^+$, and it
follows that we
have the same asymptotic results as in our theorems if we let $b=1/\E W$.

On the other hand, for $\ga=1$,
$\sum_1^n W_k=(a+\op(1))n\log n$, and for $0<\ga<1$,
$\sum_1^n W_k/n^{1/\ga}\dto Y$, where $Y$ is a stable
distribution with exponent $\ga$ (\eg see Feller~\cite[Section XVII.5]{Fe}).
It follows, arguing as in \refS{S<2}, that for $\ga=2$,
the largest clique in $\tG(n,\ga)$
has
$$(1+\op(1))\sqrt{2an}\log^{-1}n  $$
vertices,
while for $0<\ga<1$ the size of the largest clique is
always close to $\sqrt n$; more precisely,
$$\frac{\omega(\tG(n,\ga))}{\sqrt{n}\log^{-\ga/2}n}\dto
Z=a2^{\ga/2}Y^{-\ga/2},$$
where $Z$ is an absolutely continuous random variable whose
distribution has the entire positive real axis as support.
(The square $Z^2$ has, apart from a scale factor,
a Mittag-Leffler distribution with parameter $\ga$, see
Bingham, Goldie, Teugels~\cite[Section 8.0.5]{BinghamGoldieTeugels}.)
Thus, for $\ga<1$, $\omega(\tG(n,\ga))$ is not sharply concentrated around its
median; this is caused by the fact that the normalizing factor
$\sum_{i} W_i$
is determined by its first terms which, clearly, are not sharply concentrated
around their medians as well. Interestingly enough, since in
the proof of Theorem~\ref{T2} we dealt mostly
with the probability space where
we conditioned on $W_i$,
the analogue
of Theorem~\ref{T2} holds for this model as well.
Thus, for instance,
despite of the fact that neither the largest clique
nor the full top clique are
sharply concentrated in this model,
one can show the sharp
concentration result for
the ratio of these two variables.

\subsection{The model $\min(\glij,1)$}\label{subsec}
For small $\glij$, \eqref{pij} implies
$p\ij\approx\gl\ij$. In most works on
inhomogeneous random graphs, it does not matter whether we use
\eqref{pij} or, for example,
$\pij=\min(\gl\ij,1)$ or
$\pij=\gl\ij/(1+\gl\ij)$
(as in Britton, Deijfen, and Martin-L\"of~\cite{BrittonDML}),
see Bollob\'as, Janson, Riordan~\cite{SJ178}. For the cliques
studied here, however, what matters is mainly the probabilities
$p\ij$ that are close to 1, and the precise size of
$1-p\ij$ for them is important; thus it is important that
we use \eqref{pij} (cf.\ Bianconi and Marsili~\cite{BianconiM,BianconiM2}
where a cutoff is introduced).
For instance, a common version (see \eg{} \cite{SJ178}) of
$G(n,\ga)$ replaces \eqref{pij} by
\begin{equation}\label{bo}
  \pij=\min(\glij,1).
\end{equation}
This makes very little difference when $\glij$ is small,
which is the case for most $i$ and $j$, and for many
asymptotical properties the two versions are equivalent (see again
\cite{SJ178}). In the case
$\sum_iW_i^3=\op(n\qqc)$, which in our case with
$W_i$ governed by \eqref{m1} holds for
$\ga>2$ as a consequence of
\eqref{erika}, a strong general form of
asymptotic equivalence is proved in Janson~\cite{SJ212}; in the
case $\ga=2$, when $\sum_iW_i^3=\op(n\qqc)$ by
\eqref{ika}, a somewhat weaker form of equivalence (known as
contiguity) holds provided also, say,
$\max_{ij}\glij\le0.9$, see again
\cite{SJ212}.
In our case we do not need these general equivalence results; the
proofs above for the cases $\ga\ge2$ hold for this model
too, so \refT{T1}\ref{T1=2}\ref{T1>2} hold
without changes.

If $\ga<2$, however, the results are different. In fact,
\eqref{bo} implies that all vertices with $W_i\ge
b\qqw n\qq$ are joined to each other, and thus form a clique;
conversely, if we now define
$V^-=\set{i:W_i\le(b+\eps)\qqw n\qq}$,
then $\pij=\glij\le b/(b+\eps)$ for $i,j\in
V^-$, and thus $\go(G(n,\ga)[V^-])=O(\log n)$ whp by \refL{Lgnp}.
Consequently, arguing as in \refS{S<2},
\begin{equation*}
  \go(G(n,\ga))=\pca n \P(W>b\qqw n\qq)
=\pca ab\qqga n^{1-\ga/2},
\end{equation*}
so the logarithmic factor in \refT{T1}\ref{T1<2} disappears.

\subsection{The model $\glij/(1+\glij)$}
Another version of
$G(n,\ga)$ replaces \eqref{pij} by
\begin{equation}\label{bdml}
  \pij=\frac{\gl\ij}{1+\gl\ij}.
\end{equation}
This version has the interesting feature that conditioned on the
vertex degrees, the distribution is uniform over all graphs with that
degree sequence, see Britton, Deijfen, and Martin-L\"of~\cite{BrittonDML}.

In this version, for large $\glij$,
$1-p\ij=1/(1+\glij)$ is considerably larger than for
\eqref{pij} (or \eqref{bo}), and as a consequence,
the clique number is smaller.
For $\ga\ge2$, stochastic domination (or a repetition of
the proofs above) shows that
\refT{T1}\ref{T1=2}\ref{T1>2} hold without changes.

For $\ga<2$, there is a significant difference.
Arguing as in \refS{S<2}, we find that, for some constants $c$
and $C$ depending on $a$, $b$ and $\ga$, \whp
\begin{equation*}
c  n^{(2-\ga)/(2+\ga)}
\le
  \go(G(n,\ga))
\le C  n^{(2-\ga)/(2+\ga)} (\log n)^{\ga/(2+\ga)}.
\end{equation*}
Although this only determines the clique number up to a logarithmic
factor, note that the exponent of $n$ is
$\frac{2-\ga}{2+\ga}$, which is strictly
less than the exponent $\frac{2-\ga}2$ in \refT{T1}.

\subsection{Preferential attachment}
Finally, let us observe that not all power-law random graph
models contain large cliques. Indeed, one of the most popular
types of models of such graphs are preferential
attachment graphs in which the graph grows by acquiring
new vertices, where each new vertex $v$ is joined to some
number $k_v$ of `old' vertices according to some random rule
(which usually depends on the structure of the graph
we have constructed  so far), see, for instance, Durrett~\cite{Du}.
Clearly, such a graph on $n$ vertices cannot have cliques larger than
$X_n=\max_{v\le n} k_v+1$, and since for most of the models
$X_n$ is bounded from above by an absolute constant or
grows very slowly with $n$, typically the size of the largest
clique in preferential attachment random graphs is small.

\newcommand\AAP{\emph{Adv. Appl. Probab.} }
\newcommand\JAP{\emph{J. Appl. Probab.} }
\newcommand\JAMS{\emph{J. \AMS} }
\newcommand\MAMS{\emph{Memoirs \AMS} }
\newcommand\PAMS{\emph{Proc. \AMS} }
\newcommand\TAMS{\emph{Trans. \AMS} }
\newcommand\AnnMS{\emph{Ann. Math. Statist.} }
\newcommand\AnnPr{\emph{Ann. Probab.} }
\newcommand\CPC{\emph{Combin. Probab. Comput.} }
\newcommand\JMAA{\emph{J. Math. Anal. Appl.} }
\newcommand\RSA{\emph{Random Struct. Alg.} }
\newcommand\ZW{\emph{Z. Wahrsch. Verw. Gebiete} }
\newcommand\DMTCS{\jour{Discr. Math. Theor. Comput. Sci.} }

\newcommand\AMS{Amer. Math. Soc.}
\newcommand\Springer{Springer-Verlag}
\newcommand\Wiley{Wiley}

\newcommand\vol{\textbf}
\newcommand\jour{\emph}
\newcommand\book{\emph}
\newcommand\inbook{\emph}
\def\no#1#2,{\unskip#2, no. #1,} 
\newcommand\toappear{\unskip, to appear}

\newcommand\webcite[1]{
\texttt{\def~{{\tiny$\sim$}}#1}\hfill\hfill}
\newcommand\webcitesvante{\webcite{http://www.math.uu.se/~svante/papers/}}
\newcommand\arxiv[1]{\webcite{arXiv:#1.}}

\def\nobibitem#1\par{}

\end{document}